\documentclass[final,english,12pt]{article}
\usepackage{amsfonts,amsmath,amsthm,amscd,amssymb,latexsym,cite}

    \usepackage[T2A]{fontenc}
    \usepackage[cp1251]{inputenc}
    \usepackage[russian]{babel}

\usepackage[notref,notcite]{showkeys}
\usepackage{amsfonts}
\usepackage{amssymb}
\usepackage{stmaryrd}
\usepackage{textcomp}
\usepackage{amsthm}
\usepackage{amsmath}

\usepackage{ifpdf}
\ifpdf
\usepackage[pdftex]{graphicx}
\usepackage[pdftex,dvipsnames,usenames]{color}
\else
\usepackage{graphicx}
\usepackage[dvips,dvipsnames,usenames]{color}
\fi

\newtheorem{theorem}{Теорема}
\newtheorem{lemma}[theorem]{Лемма}
\newtheorem{definition}[theorem]{Определение}
\newtheorem{corollary}[theorem]{Следствие}
\newtheorem{statement}[theorem]{Утверждение}

\theoremstyle{definition}
\newtheorem{remark}[theorem]{Замечание}

\begin{document}

%\eng
%Если текст статьи английский, то знак процента перед \eng нужно убрать.

УДК 515.124
% Индекс УДК

\begin{center}
{\Large Ball-preserving mappings of finite ultrametric spaces}

E. Petrov
\end{center}

It is shown that the rooted trees $T_X$ and $T_Y$ representing
finite ultrametric spaces $X$ and $Y$ are isomorphic if and only if
there exists a ball-preserving bijection $F:X\to Y$.



\begin{center}
{\Large Сохраняющие шары отображения конечных  ультраметрических  пространств}

Е.А. Петров
\end{center}

\begin{abstract}
Показано, что корневые представляющие деревья $T_X$ и $T_Y$ конечных ультраметрических пространств $X$ и $Y$
изоморфны тогда и только тогда, когда существует сохраняющая шары биекция $F:X\to Y$.
\vspace{1mm}\\
\textbf{\emph {Ключевые слова: }}\itshape{конечное ультраметрическое пространство, конечное корневое дерево, сохраняющее шары отображение, боллеан.}
\end{abstract}

\section{Введение} В 2001 году на семинаре Workshop on General Algebra (см.~\cite{Lemin(2001)}) внимание специалистов по теории решёток было обращено на следующую задачу И.М. Гельфанда. \emph{Используя теорию графов, описать с точностью до изометрии все конечные ультраметрические пространства}. В~\cite{Lemin(2003)}  была доказана теорема про изоморфизм категории ультраметрических пространств и категории полных, атомных, древовидных, градуированных действительными числами решеток. В работе~\cite{GurVyal(2012)} взвешенный орграф определял квазиультраметрическое конечное пространство, а в симметричном случае авторы получали ``каноническое представление'' конечных ультраметрических пространств с использованием взвешенных корневых деревьев, причём эти деревья были изоморфны как корневые взвешенные графы тогда и только тогда, когда соответствующие им ультраметрические пространства были изометричными.

Каноническое представление из~\cite{GurVyal(2012)} можно, в определённом смысле, считать решением упомянутой выше задачи И.М. Гельфанда. Естественно возникает вопрос о применении полученного представления в исследовании ультраметрических пространств. В связи с этим заметим, что в последнее время началось изучение боллеанов (ballean) метрических и более общих пространств (см., например,~\cite{ProtBanakh(2003)}, ~\cite{ProtZarich(2007)}). В метрическом случае боллеан~--- это просто совокупность шаров пространства и исследование боллеанов очевидным образом связано с изучением класса отображений, сохраняющих свойство ``быть шаром''.

В работе ~\cite{PD} изучался класс конечных ультраметрических пространств являющихся ``экстремалями'' фундаментального неравенства Гомори-Ху. Для  пространств $(X,d)$, $(Y,\rho)$ из этого класса было доказано, что их представляющие деревья $T_X$ и $T_Y$ изоморфны тогда и только тогда, когда существует сохраняющее шары биективное отображение $F:X\to Y$. Таким образом, вопрос об изоморфизме боллеанов пространств $X$ и $Y$ оказался эквивалентным вопросу об изоморфизме корневых деревьев, получаемых путём ``забывания'' заданных на них весовых функций. В настоящей работе мы обобщаем этот результат на случай произвольных конечных ультраметрических пространств.

Напомним необходимые определения. Для конечного множества $X$ через $|X|$ будем обозначать количество его элементов. Пусть $(X,d)$~--- метрическое пространство. Если метрика $d$ удовлетворяет сильному неравенству треугольника
$$
d(x,y)\leq \max \{d(x,z),d(z,y)\}
$$
для всех $x,y,z \in X$, то она называется \emph{ультраметрикой},  а пара $(X,d)$ называется \emph{ультраметрическим пространством}. В дальнейшем будем рассматривать только те пространства $(X,d)$, для которых $|X|\neq 0$. \emph{Диаметром} метрического пространства $(X,d)$ называется  величина
$$
\mathrm{diam} X:=\sup\{d(x,y):x,y \in X\}.
$$

Под \textit{графом} мы  понимаем пару $(V,E)$, состоящую из непустого множества $V$ и (возможно пустого) множества $E$, элементы которого есть неупорядоченные пары различных точек из $V$. Для графа $G=(V,E)$ множество $V=V(G)$ называется \textit{множеством вершин}, а $E=E(G)$ ~--- \textit{множеством рёбер}. Если элементами множества $E$ являются упорядоченные пары $\langle x,y\rangle\in V\times V$, то $G=(V,E)$~--- ориентированный граф (\emph{орграф}).

Граф $H$ является {\it подграфом} графа $G$, $H\subseteq G$, если $V(H)\subseteq V(G)$ и $E(H)\subseteq E(G)$. Граф $G$ {\it конечен}, если $|V(G)|<\infty$. Если $E(G)=\varnothing$, то $G$~--- \textit{пустой граф}. Конечный непустой граф $P\subseteq G$ называется \textit{путём} (в $G$), если вершины из $P$ можно без повторений занумеровать в последовательность $(v_1,v_2,...,v_n)$ так, что $ (\{v_i,v_j\}\in E(P))\Leftrightarrow (|i-j|=1)$. Две вершины в графе \textit{связаны}, если существует соединяющий их путь. {\it Связный граф} — граф, в котором все вершины связаны.

\textit{Деревом} называется связный граф, не имеющий циклов.  Выбранная вершина дерева называется \emph{корнем дерева}. Дерево, содержащее такую вершину, называется \emph{корневым деревом}. Вершину дерева иногда называют \emph{узлом.} \emph{Уровень узла} — длина пути от корня до узла, $m$-й \emph{ярус} дерева~--- множество узлов дерева, на уровне $m$ от корня дерева. \emph{Потомками} данного узла будем называть все узлы последующего яруса, смежные с данным узлом. \emph{Лист} дерева~--- вершина дерева инцидентная с единственным ребром. Если $T$~--- корневое дерево с единственным узлом, то мы считаем, что этот узел является листом.  \emph{Внутренний узел}~--- узел дерева, не являющийся листом.  Два корневых дерева $T_1$ и $T_2$ называются изоморфными, если существует биекция $F:V(T_1)\to V(T_2)$, переводящая корень дерева $T_1$  в корень дерева $T_2$ и такая, что
$$
(\{u,v\}\in E(T_1))\Leftrightarrow (\{F(u),F(v)\}\in E(T_2))
$$
для любых различных $u,v \in V(T_1)$.

Граф $G$ называется \emph{полным $k$-дольным}, если его вершины можно разбить на $k$ непустых непересекающихся подмножества $X_1,...,X_k$ так, что нет рёбер, соединяющих вершины одного и того же $X_i$, и две любые вершины из разных $X_i, X_j$, $1\leqslant i, j\leqslant k$ смежны. В этом случае пишем $G=G[X_1,...,X_k]$.

Пусть $(Y,\leqslant_Y)$~--- конечное частично упорядоченное множество. Под \emph{диаграммой Хассе} ч.у. множества $(Y,\leqslant_Y)$ мы понимаем орграф с множеством вершин $Y$ и множеством дуг  (ориентированных рёбер) $A_Y\subseteq Y\times Y$ таких, что  пара $\langle v_1,v_2 \rangle$ принадлежит $A_Y$ тогда и только тогда, когда $v_1\leqslant_Y v_2$, $v_1\neq v_2$ и импликация
$$
(v_1\leqslant_Y w\leqslant_Y v_2)\Rightarrow (v_1=w \vee v_2=w)
$$
имеет место для любого $w\in Y$. Два орграфа $(Y,A_Y)$ и $(X,A_X)$ являются изоморфными, если существует биекция $F:X\to Y$ такая, что
\begin{equation}\label{d1}
  (\langle x,y\rangle\in A_X)\Leftrightarrow(\langle F(x),F(y)\rangle \in A_Y),
\end{equation}
в этом случае $F$~--- изоморфизм орграфов $(Y,A_Y)$ и $(X,A_X)$.
Любое корневое дерево $T$ можно рассматривать как орграф $(V(T),A_T)$, если положить
\begin{equation}\label{d2}
  (\langle u,v \rangle \in A_T) \Leftrightarrow  (u \text{~--- потомок } v).
\end{equation}
Следующие утверждения почти очевидны.

\begin{statement}\label{prop1}
  Пусть $(X,\leqslant_X)$, $(Y,\leqslant_Y)$~--- конечные ч.у. множества, $(X,A_X)$, $(Y,A_Y)$~--- соответствующие им диаграммы Хассе, $F:X\to Y$~--- биекция. Отображение $F$ является изоморфизмом ч.у. множеств $(X,\leqslant_X)$ и $(Y,\leqslant_Y)$ тогда  и только тогда, когда оно является изоморфизмом орграфов $(X,A_X)$ и $(Y,A_Y)$.
\end{statement}

В этом утверждении и далее изоморфизм  ч.у. множеств понимаем в стандартном смысле (см., например,~\cite[стр. 44]{MR(1990)}).

\begin{statement}\label{prop2}
Пусть $T_1$ и $T_2$~--- конечные корневые деревья, $X=V(T_1)$, $Y=V(T_2)$ и $(X,A_X)$, $(Y,A_Y)$~--- орграфы, соответствующие $T_1$ и $T_2$, $F:X\to Y$~--- биекция. Отображение $F$ является изоморфизмом корневых деревьев $T_1$ и $T_2$ тогда и только тогда, когда оно является изоморфизмом орграфов $(X,A_X)$ и $(Y,A_Y)$.
\end{statement}

\section{Представляющие деревья} Пусть $(X,d)$~--- ультраметрическое пространство. \emph{Диаметральным графом}  $G_{d}$ будем называть граф, для которого
$$
V(G_{d})=X \, \text{ и } \, (\{u,v\}\in E(G_{d}))\Leftrightarrow (d(u,v)=\mathrm{diam}X).
$$

Напомним, что в метрическом пространстве $(X,d)$ замкнутым шаром радиуса $r$ с центром в точке $t\in X$ называется множество
$$
B_r(t)=\{x\in X:d(x,t)\leqslant r\}.
$$
Для каждого $t\in X$ положим $\mathrm{Sp}_t(X):=\{d(x,t):x\in X\}$. Обозначим через $\textbf{B}_X$ множество всех шаров $B_r(t)$ с $r\in \mathrm{Sp}_t(X)$, т.е.
$$
\textbf{B}_X=\{B_r(t):t\in X, r\in \mathrm{Sp}_t(X)\}.
$$

Нам понадобится следующая теорема из~\cite{DDP(P-adic)}

\begin{theorem}\label{th1}
Пусть $(X,d)$~--- конечное ультраметрическое про\-стран\-ство с $|X|\geqslant 2$.  Тогда $G_d=G_d[X_1,...,X_k]$, $k\geqslant 2$.
 \end{theorem}

Для каждого $i=1,...,k$ рассмотрим ультраметрические пространства $(X_i,d)$, где $X_i$~--- подмножество множества $X$ из предыдущей теоремы с метриками, полученными сужением ультраметрики $d$ на $X_i$. Пусть $d_i=\mathrm{diam}X_i$ и $x_i\in X_i$. Очевидно, $d(x_i,y)\leqslant d_i < \mathrm{diam}X$ для всех $y\in X_i$, а при $y\in X\backslash X_i$ имеем $d(x_i,y)=\mathrm{diam}X > d_i$. Таким образом, выполнена следующая

\begin{lemma}\label{lem2}
Пусть $(X,d)$~--- конечное ультраметрическое про\-стран\-ство с $|X|\geqslant 2$ и диаметральным графом $G_d[X_1,...,X_k]$. Тогда имеет место принадлежность $X_i\in \textbf{B}_X$, $1\leqslant i\leqslant k$.
 \end{lemma}

Как это было сделано в~\cite{PD}, поставим каждому конечному ультраметрическому пространству $(X,d)$ в соответствие помеченное корневое дерево $T_X$ по следующему правилу. Если $X=\{x\}$~--- одноточечное множество, то $T_X$~--- дерево, состоящее из одного узла, помеченного меткой $x$, которое мы считаем корневым по определению. Пусть $|X|\geqslant 2$. Корень $v_0$ дерева пометим меткой $\bar{v}_0=\mathrm{diam} X$. Пусть $G_d$~--- диаметральный граф пространства $(X,d)$. По теореме~\ref{th1} $G_d=G_d[X_1,...,X_k]$.  В этом случае будем считать, что дерево $T_X$ имеет $k$ узлов $v_1,v_2,...,v_k$, лежащих на первом ярусе с метками
\begin{equation}\label{eq1.6}
\bar{v}_i:=
\begin{cases}
\mathrm{diam} X_i, &\text{если } |X_i|\geqslant 2\\
x,  &\text{если } X_i \text{~--- одноточечное множество} \\
&\text{с единственным элементом } x,
\end{cases}
\end{equation}
$i=1,...,k$. Узлы первого яруса, помеченные метками $x\in X$, будут листьями, а метками $\mathrm{diam} X_i$~--- внутренними узлами дерева $T_X$. Если на первом ярусе внутренних узлов нет, то дерево $T_X$ построено. В противном случае, повторяя описанную выше процедуру с пространствами $(X_i,d)$, соответствующими внутренним узлам первого яруса, получаем узлы второго яруса и т.д.. Так как $|X|$ конечно, то на каком-то из ярусов все вершины будут листьями и построение дерева $T_X$ завершается.

Построенное выше помеченное корневое дерево $T_X$ будем называть \emph{представляющим деревом пространства} $(X,d)$. Отметим, что разным листьям соответствуют разные $x\in X$ и каждый элемент $x\in X$ приписан какому-то листу, но различные внутренние узлы могут иметь совпадающие метки.
В дальнейшем изложении мы будем отождествлять листья дерева $T_X$ с их метками, если это удобно.

\begin{remark}\label{rem1}
Пусть $|X|\geqslant 2$ и $(v_0, v_1,...,v_n,x_i)$~--- путь от корня $v_0$ дерева $T_X$ до произвольного листа $x_i$, тогда $\mathrm{diam}X=\bar{v}_0>\bar{v}_1>...>\bar{v}_n$.
\end{remark}

Следующая лемма была сформулирована в~\cite{PD} для специального класса конечных ультраметрических пространств, но её доказательство справедливо для всех таких пространств.
\begin{lemma}\label{lem14}
Пусть $(X,d)$~--- конечное ультраметрическое пространство и пусть $x_1$, $x_2$~--- два различных листа дерева $T_X$. Тогда, если $(x_1,v_1,...,v_n,x_2)$~--- путь, соединяющий листья $x_1$ и $x_2$ в $T_X$, то
\begin{equation}\label{eq10}
d(x_1,x_2)=\max_{1\leqslant i\leqslant n} \bar{v}_i.
\end{equation}
\end{lemma}

\begin{lemma}\label{lem4'}
Пусть $(X,d)$~--- ультраметрическое пространство и $Y,Z \subseteq X$. Тогда,  если $Y\in \textbf{B}_X$ и $Z \in \textbf{B}_Y$, то $Z\in \textbf{B}_X$.
\end{lemma}

\begin{proof}
Выберем $y_1\in Y$, $z_1\in Z$ и $r_y, r_z\in [0,\infty)$ так, что $Y=\{x\in X:d(y_1,x)\leqslant r_y\}$ и $Z=\{y\in Y:d(z_1,y)\leqslant r_z\}$. Так как в ультраметрическом пространстве любая точка шара является его центром, то
\begin{equation}\label{a0}
Y=\{x\in X:d(z_1,x)\leqslant r_y\}.
\end{equation}
Диаметры шаров из $\textbf{B}_X$ и $\textbf{B}_Y$ совпадают с их радиусами, значит
\begin{equation}\label{a1}
r_y=\mathrm{diam} Y, \quad r_z=\mathrm{diam} Z,
\end{equation}
а так как  $Z\subseteq Y$, то из~(\ref{a1}) следует
\begin{equation}\label{a2}
r_y\geqslant r_z.
\end{equation}
Принадлежность $Z\in \textbf{B}_X$ равносильна тому, что
\begin{equation}\label{a3}
Z=\{x\in X:d(x,z_1)\leqslant r_z\}.
\end{equation}
Включение $Z\subseteq\{x\in X:d(x,z_1)\leqslant r_z\}$ очевидно, поэтому достаточно доказать обратное включение. Пусть $x_0\in X$ и $d(z_1,x_0)\leqslant r_z$. Тогда в силу~(\ref{a2}) и~(\ref{a0}) имеем $x_0\in Y$. Отсюда и неравенства $d(z_1,x_0)\leqslant r_z$ следует, что $x_0\in Z$. Равенство~(\ref{a3}) доказано.
\end{proof}

\begin{definition}\label{def6*}
Пусть $T$~--- конечное корневое дерево с корнем $v_0$. Для каждой вершины $v\in V(T)$ определим подграф $T^v$ следующим образом. Если $v=v_0$, то $T^v:=T$, Если $v\neq v_0$, то пусть $u$~--- единственная вершина $T$ такая, что $v$~--- потомок $u$. Рассмотрим $G\subseteq T$ с
$$
V(G):=V(T)\, \text{ и } \, E(G):=E(T)\backslash \{u,v\}.
$$
Граф $G$ представляет собой лес, состоящий из двух деревьев. $T^v$~--- то из этих деревьев, которое содержит $v$.
\end{definition}

Пусть как в определении~\ref{def6*} $T$~--- конечное корневое дерево с корнем $v_0$. Обозначим через $L_v$ множество листьев графа $T^v$. Легко показать, что $L_v\subseteq L$, где $L$~--- множество листьев графа $T$. Рассмотрим многозначное отображение
\begin{equation}\label{b1}
\Gamma_T:V(T)\to 2^{L}
\end{equation}
такое, что $\Gamma_T(v)=L_v$ для $v\in V(T)$. Заметим, что если $v$~--- лист дерева $T$, то $\Gamma_T(v)=\{v\}$, где $\{v\}$~--- одноточечное множество, состоящее из единственного элемента $v$.

Исследуем отображение $\Gamma_T$ в случае когда $T$ является представляющим деревом конечного ультраметрического пространства $X$.

\begin{lemma}\label{lem7}
Пусть $(X,d)$~--- конечное ультраметрическое пространство с представляющим деревом $T_X$, $|X|\geqslant 2$. Тогда
\begin{itemize}
  \item [(i)] отображение $\Gamma_{T_X}:V(T_X)\to 2^{X}$ является инъективным,
  \item [(ii)] для любого $v\in V(T_X)$ имеет место принадлежность $\Gamma_{T_X}(v)\in \textbf{B}_X$,
  \item [(iii)] для любого $\tilde{B} \in \textbf{B}_X$ существует узел $\tilde{v}$ такой, что $\Gamma_{T_X}(\tilde{v})=\tilde{B}$.
\end{itemize}
\end{lemma}
\begin{proof}
Инъективность $\Gamma_{T_X}$ и принадлежность $\Gamma_{T_X}(v) \in \textbf{B}_X$ следует из леммы~\ref{lem4'}, леммы~\ref{lem2}  и приведённого выше построения представляющего дерева $T_X$. Покажем, что для любого $\tilde{B}=\{x_1,..,x_k\}\in \textbf{B}_X$ найдётся узел $\tilde{v}\in V(T_X)$ такой, что $\Gamma_{T_X}(\tilde{v})=\tilde{B}$. При $k=1$ это очевидно, поэтому считаем $k\geqslant 2$. Положим $b:=\max\{d(x,y), x,y\in \tilde{B}\}$ и пусть $b=d(x_i,x_j)$. Т.к. $\tilde{B}$~--- это шар в ультраметрическом пространстве, то он совпадает с множеством $\{x \in X: d(x,x_i)\leqslant b\}$. Пусть $(x_i,v_1,...,v_n,x_j)$~--- путь, соединяющий листья $x_i$ и $x_j$ в $T_X$. По лемме~\ref{lem14} $b=d(x_i,x_j)=\max\limits_{1\leqslant i \leqslant n}\bar{v}_i$. В силу замечания~\ref{rem1}, узел $\tilde{v}$, помеченный меткой $b$, будет узлом наименьшего яруса среди узлов $v_i$. Рассмотрим корневое поддерево $T^{\tilde{v}}_X$ дерева $T_X$. Покажем, что множество листьев $L_{\tilde{v}}$ поддерева $T^{\tilde{v}}_X$ совпадает с множеством $\tilde{B}$. Пусть $x\in L_{\tilde{v}}$. Рассмотрим путь $(x,v_1,...,v_n,x_i)$. Очевидно, что $d(x,x_i)=\max\limits_{1\leqslant i\leqslant n}\bar{v}_i\leqslant b$. Следовательно, $x\in \tilde{B}$. Обратно, пусть $x\in \tilde{B}$ и предположим, что $x\notin L_{\tilde{v}}$. Рассмотрим путь $(x,v_1,..,v_n,x_i)$ в $T_X$, соединяющий $x$ и $x_i$. Тогда вершина $\tilde{v}$ является потомком одной из вершин $v_i$ этого пути, что даёт неравенство $b<\bar{v}_i$. Следовательно, $d(x,x_i)=\max\limits_{1\leqslant i\leqslant n}\bar{v}_i>b$, что противоречит принадлежности $x\in \tilde{B}$.
\end{proof}

Корневое дерево с корнем $v_0$ (см.~(\ref{eq1.6})), получающееся из $T_X$ путём ``стирания меток''  будем обозначать через $\overline{T}_X$.

\begin{lemma}\label{lem8}
Пусть $(X,d)$~--- конечное ультраметрическое пространство,   $u, v \in V(\overline{T}_X)$, $u\neq v$. Узел $u$ является потомком узла $v$ тогда и только тогда, когда
\begin{equation}\label{c1}
\Gamma_{\overline{T}_X}(u)\subseteq \Gamma_{\overline{T}_X}(v)
\end{equation}
и импликация
\begin{equation}\label{c2}
(\Gamma_{\overline{T}_X}(u)\subseteq \Gamma_{\overline{T}_X}(w)\subseteq \Gamma_{\overline{T}_X}(v))\Rightarrow (\Gamma_{\overline{T}_X}(u)=\Gamma_{\overline{T}_X}(w))\vee(\Gamma_{\overline{T}_X}(w)=\Gamma_{\overline{T}_X}(v))
\end{equation}
выполняется для любого $w\in V(\overline{T}_X)$.
\end{lemma}
Доказательство этой леммы достаточно просто и мы его опускаем.

Из лемм~\ref{lem7} и~\ref{lem8} выводится
\begin{corollary}\label{cor8*}
Пусть $(X,d)$~--- конечное ультраметрическое пространство. Если на $\textbf{B}_X$ задать частичный порядок, индуцированный из частично упорядоченного множества $(2^X,\subseteq)$, то диаграмма Хассе ч.у. множества $(\textbf{B}_X, \subseteq)$ изоморфна орграфу  $(V(\overline{T}_X),A_{\overline{T}_X})$.
\end{corollary}

\section{Отображение, сохраняющее шары} Сформулируем  центральное для этой работы

\begin{definition}\label{def12*}
Пусть $X$ и $Y$~--- метрические пространства. Отображение $F:X\to Y$ сохраняет шары, если для любых $Z\in \textbf{B}_X$ и $W\in \textbf{B}_Y$ выполнены соотношения
\begin{equation}\label{eq9}
F(Z)\in \textbf{B}_Y \,\text{ и } \, F^{-1}(W)\in \textbf{B}_X,
\end{equation}
где $F(Z)$~--- образ множества $Z$ при отображении $F$ и $F^{-1}(W)$~---  прообраз множества $W$ при этом отображении.
\end{definition}

Отметим, что для любого биективного отображения $F:X\to Y$ и любых подмножеств $X_1$, $X_2$ множества $X$ включение $X_1\subseteq X_2$ имеет место тогда и только тогда, когда $F(X_1)\subseteq F(X_2)$. Следовательно справедлива следующая

\begin{lemma}\label{lem9*}
Пусть $X,Y$~--- метрические пространства и $F:X\to Y$~--- сохраняющее шары биективное отображение. Тогда для любых $B_1, B_2\in \textbf{B}_X$ имеет место эквивалентность
$$
(B_1\subseteq B_2)\Leftrightarrow (F(B_1)\subseteq F(B_2)).
$$
\end{lemma}

\begin{corollary}\label{cor10*}
Пусть $X,Y$~--- метрические пространства и $F:X\to Y$~--- сохраняющее шары биективное отображение. Тогда отображение
$$
(\textbf{B}_X,\subseteq)\ni B\mapsto F(B)\in(\textbf{B}_Y,\subseteq)
$$
есть изоморфизм ч.у. множеств $(\textbf{B}_X,\subseteq)$ и $(\textbf{B}_Y,\subseteq)$.
\end{corollary}

\begin{theorem}\label{th13*}
Пусть $X$, $Y$~--- конечные ультраметрические пространства. $\overline{T}_X$ и $\overline{T}_Y$ изоморфны как корневые деревья тогда и только тогда, когда существует сохраняющее шары биективное отображение $\Phi:X\to Y$.
\end{theorem}

\begin{proof} Теорема тривиальна при $|X|=1$, поэтому будем считать, что $|X|\geqslant 2$. Пусть $V_X = V(\overline{T}_X)$ и $V_Y = V(\overline{T}_Y)$~--- множества вершин графов $\overline{T}_X$ и $\overline{T}_Y$ соответственно. Предположим, что существует биективное отображение $\Psi:V_X\to V_Y$, сохраняющее отношение смежности между вершинами и переводящее корень дерева $\overline{T}_X$ в корень дерева $\overline{T}_Y$.
Биекция $\Psi$ отображает множество листьев графа $\overline{T}_X$~--- множество $X$ на множество листьев графа $\overline{T}_Y$~--- множество $Y$, так как листья ~--- это в точности вершины степени 1. Обозначим через $\Phi$ сужение $\Psi$  на $X$, $\Phi=\Psi|_{X}$, и покажем, что биективное отображение $\Phi:X\to Y$, рассматриваемое как отображение между ультраметрическими пространствами $(X,d)$ и $(Y,\rho)$, сохраняет шары.

Пусть $B\in \textbf{B}_X$. Покажем что
\begin{equation}\label{b2}
\Phi(B)\in \textbf{B}_Y
\end{equation}
В силу леммы~\ref{lem7} существует узел $v$ дерева $\overline{T}_X$, для которого $B$ совпадает с множеством листьев $L_v$ графа $\overline{T}_X^v$, где $\overline{T}_X^v$~--- поддерево дерева $\overline{T}_X$ задаваемое определением~\ref{def6*}. Так как $\Psi$~--- изоморфизм, то образом дерева $\overline{T}_X^v$ является какое-то поддерево $T'$ дерева $\overline{T}_Y$. Пусть $u=\Psi(v)$ и $\overline{T}_Y^u$~--- поддерево дерева $\overline{T}_Y$, построенное для $u$ в соответствии с определением~\ref{def6*}. Так как при изоморфизме $\Psi$ корень дерева $\overline{T}_X$ переходит в корень дерева $\overline{T}_Y$, то используя определение~\ref{def6*}, легко установить равенство $T'=\overline{T}_X^u$. Сужение $\Psi$ на $V(\overline{T}_X^v)$ является изоморфизмом деревьев $\overline{T}_X^v$ и $\overline{T}_Y^u$. При изоморфизме множество листьев переходит в множество листьев. Пусть $L_u$~--- множество листьев дерева $\overline{T}_Y^u$. Тогда
\begin{equation}\label{b3}
L_u=\Psi|_{V(\overline{T}_X^v)}(L_v)=\Psi(L_v).
\end{equation}
Так как $L_v\subseteq X$, а $\Phi=\Psi|_X$, то из ~(\ref{b3}) получаем
\begin{equation}\label{b4}
L_u=\Phi(L_v)=\Psi(B).
\end{equation}
По лемме~\ref{lem7} имеем $L_u\in \textbf{B}_Y$. Отсюда и из~(\ref{b4}) следует~(\ref{b2}). Аналогично устанавливается, что
$$
\Phi^{-1}(Z)\in \textbf{B}_X
$$
для любого $Z\in \textbf{B}_Y$. Таким образом, из того, что $\overline{T}_X$ и $\overline{T}_Y$ изоморфны как корневые деревья следует, что $\Phi$~--- сохраняющая шары биекция.

Пусть теперь \mbox{$\Phi:X\to Y$} сохраняющая шары биекция. Нужно доказать, что $\overline{T}_X$ и $\overline{T}_Y$ изоморфны как корневые деревья.
Пусть $(Y,A_Y)$, $(X,A_X)$~--- диаграммы Хассе ч.у. множеств $(\textbf{B}_Y,\subseteq)$, $(\textbf{B}_X,\subseteq)$ и пусть $(V(\overline{T}_Y), A_{\overline{T}_Y})$, $(V(\overline{T}_X), A_{\overline{T}_X})$ орграфы, соответствующие  $\overline{T}_Y$, $\overline{T}_X$. В соответствии со следствием~\ref{cor8*}, $(X,A_X)$ и $(V(\overline{T}_X), A_{\overline{T}_X})$~--- изоморфны как ориентированные графы, аналогично, орграфы $(Y,A_Y)$ и $(V(\overline{T}_Y), A_{\overline{T}_Y})$ тоже изоморфны. Используя утверждение ~\ref{prop1} и следствие~\ref{cor10*}, убеждаемся в изоморфизме орграфов $(Y,A_Y)$ и $(X,A_X)$. Следовательно орграфы $(V(\overline{T}_X), A_{\overline{T}_X})$ и $(V(\overline{T}_Y), A_{\overline{T}_Y})$ тоже изоморфны. Последнее по утверждению~\ref{prop2} равносильно изоморфности корневых деревьев $\overline{T}_X$ и $\overline{T}_Y$.
\end{proof}

Рассуждения, аналогичные проведённым во второй части доказательства теоремы~\ref{th13*},  показывают, что $\overline{T}_X$ и $\overline{T}_Y$ изоморфны как корневые деревья тогда и только тогда, когда изоморфны ч.у. множества $(\textbf{B}_X,\subseteq)$ и $(\textbf{B}_Y,\subseteq)$. Таким образом, имеет место
\begin{corollary}\label{cor17}
Пусть $X$ и $Y$~--- конечные ультраметрические пространства. Ч.у. множества $(\textbf{B}_X,\subseteq)$ и $(\textbf{B}_Y,\subseteq)$ изоморфны тогда и только тогда, когда существует сохраняющая шары биекция $\Phi:X\to Y$.
\end{corollary}

\end{document}